\documentclass[12pt]{amsart}

\usepackage{amssymb}

\usepackage{enumitem}

\usepackage{hyperref}

\makeatletter
\@namedef{subjclassname@2020}{%
  \textup{2020} Mathematics Subject Classification}
\makeatother

\usepackage[T1]{fontenc}

\newtheorem{theorem}{Theorem}[section]
\newtheorem{corollary}[theorem]{Corollary}
\newtheorem{lemma}[theorem]{Lemma}
\newtheorem{conjecture}[theorem]{Conjecture}
\newtheorem{proposition}[theorem]{Proposition}

\newtheorem{fact}[theorem]{Fact}

\theoremstyle{definition}
\newtheorem{definition}[theorem]{Definition}
\newtheorem{remark}[theorem]{Remark}

\numberwithin{equation}{section}

\begin{document}


\baselineskip=17pt


\title{Ill-distributed sets over global fields and exceptional sets in Diophantine Geometry}

\author{Marcelo Paredes}
\address{
Instituto Argentino de Matem\'aticas-CONICET\\
Saavedra 15, Piso 3 (1083), Buenos Aires, Argentina}
\address{
and}
\address{Departamento de Matem\'aticas, Facultad de Ciencias Exactas y Naturales\\
Universidad de Buenos Aires, 1428\\
Buenos Aires\\
Argentina}
\email{mparedes@dm.uba.ar}



\begin{abstract}	
Let $K\subseteq \mathbb{R}$ be a number field. Using techniques of discrete analysis, we prove that for definable sets $X$ in $\mathbb{R}_{\exp}$ of dimension at most $2$ a conjecture of Wilkie about the density of rational points is equivalent to the fact that $X$ is badly distributed at the level of residue classes for many primes of $K$. This provides a new strategy to prove this conjecture of Wilkie. In order to prove this result, we are lead to study an inverse problem as in the works \cite{Walsh2, Walsh}, but in the context of number fields, or more generally global fields. Specifically, we prove that if $K$ is a global field, then every subset $S\subseteq \mathbb{P}^{n}(K)$ consisting of rational points of projective height bounded by $N$, occupying few residue classes modulo $\mathfrak{p}$ for many primes $\mathfrak{p}$ of $K$, must essentially lie in the solution set of a polynomial equation of degree $\lesssim (\log(N))^{C}$, for some constant $C$. 
\end{abstract}

\maketitle

\section{Introduction}

For $K$ a number field, and $X\subseteq \mathbb{R}^{n}$, let $X(K)$ denote the subset of points ${\bf x}=(x_{1},\ldots ,x_{n})\in X$  with $x_{i}\in K$ for all $i$. Given $x\in K$, let $H(x)$ be the affine height of an algebraic number. For $N\geq 1$, set
\begin{center}
$X(K,N):=\{{\bf x}=(x_{1},\ldots ,x_{n})\in X(K):H(x_{i})\leq N\; \forall i\}$.
\end{center}
A fundamental problem in Diophantine Geometry and Transcendental Number Theory is to obtain bounds for $X(K,N)$ when $X$ is a non-algebraic set. When $X$ is the graph of $f:[0,1]\rightarrow \mathbb{R}$, a transcendental real-analytic function, in \cite[Theorem 9]{Pila5} Pila  proves that for any $\varepsilon>0$ there exists a positive constant $c=c(X,\varepsilon)$ such that
\begin{equation}
|X(\mathbb{Q},N)|\leq cN^{\varepsilon}.
\label{Pila5}
\end{equation}
In order to generalize \eqref{Pila5} to sets of higher dimensions, Pila and Wilkie in \cite{Pila2} deal with the transcendental part of a set $X\subseteq \mathbb{R}^{n}$ definable in an o-minimal structure. Recall that the algebraic part of a set $X\subseteq \mathbb{R}^{n}$, which we denote $X^{\text{alg}}$, consists of the points ${\bf x}\in X$ such that there exists a connected, semialgebraic set $Y\subseteq X$ of positive dimension with ${\bf x}\in Y$. The transcendental part of $X$, denoted $X^{\text{trans}}$, is defined as $X^{\text{trans}}:=X\backslash X^{\text{alg}}$. Pila and Wilkie then prove the following generalization of \eqref{Pila5}.
\begin{theorem}[{\cite[Theorem 1.8]{Pila2}}]
Let $X\subseteq \mathbb{R}^{n}$ be a set, definable in an o-minimal structure, and let $\varepsilon>0$. There is a positive constant $c=c(X,\varepsilon)$ such that for all $N\geq 1$ we have
\begin{equation}
|X^{{\rm trans}}(\mathbb{Q},N)|\leq cN^{\varepsilon}.
\label{pil}
\end{equation}
\label{pilawilkietheorem}
\end{theorem}
Theorem \ref{pilawilkietheorem} was later generalized by Pila in \cite{Pila}. From the results in \cite{Pila}, it follows that the same type of bound \eqref{pil} holds for number fields, but with the constant $c$ dependent on the (degree of the) field.

In general the bound in Theorem \ref{pilawilkietheorem}  is best possible, but if we consider some specific o-minimal structures, it is conjectured that the bound can be improved. This is the content of Wilkie conjecture:
\begin{conjecture}[Wilkie conjecture, {\cite[Conjecture 1.11]{Pila2}}]
Suppose that $X\subseteq \mathbb{R}^{n}$ is a set definable in the o-minimal structure $\mathbb{R}_{{\rm exp}}$. For any number field $K\subseteq \mathbb{R}$ of degree $d$, there exist positive constants $c_{1}=c_{1}(X,d)$, $c_{2}=c_{2}(X)$ such that
\begin{equation}
\left\vert X^{{\rm trans}}(K,N) \right\vert\leq c_{1}\left(\log(N)\right)^{c_{2}}
\end{equation}
for all $N>e$.
\label{wilkie}
\end{conjecture}
Let us note that Conjecture \ref{wilkie} has consequences in Transcendental Number Theory. For instance, let $\alpha\in \mathbb{R}\backslash \mathbb{Q}$, and consider the set 
$$X_{\alpha}:=\{(x,y)\in \mathbb{R}^{2}:y=x^{\alpha}\}.$$ 
It can be seen that $X_{\alpha}$ is definable in $\mathbb{R}_{\exp}$, and that it does not contain semialgebraic subsets of positive dimension. Let us suppose that $X_{\alpha}$ verifies Conjecture \ref{wilkie} for all number field $K$, with constant $c_{2}$. Using the group structure of $X_{\alpha}$, it is not hard to prove (see \cite[Page 493]{Pila4}) that given $\left\lceil c_{2}\right\rceil+1$ real numbers $w_{1},\ldots ,w_{\left\lceil c_{2}\right\rceil+1}$, linearly independent over $\mathbb{Q}$, then at least one of the $2(\left\lceil c_{2}\right\rceil+1)$ exponentials $\exp (w_{1}),\ldots ,\exp(w_{\left\lceil c_{2}\right\rceil+1})$, $\exp(\alpha w_{1}),\ldots ,\exp(\alpha w_{\left\lceil c_{2}\right\rceil+1})$ must be transcendental. 

In fact, the Six Exponentials Theorem states that if there are just $3$ linearly independent $w_{i}$, then at least one of the six exponentials $\exp(w_{i})$, $\exp(\alpha w_{i})$, $i=1,2,3$, is transcendental. Moreover, it is conjectured that if we change $3$ by $2$, the same result holds; this is the Four Exponentials Conjecture. Thus, if the set $X_{\alpha}$ verifies Conjecture \ref{wilkie} with constant $c_{2}=1$, the Four Exponentials Conjecture would follow. For more about this, see also \cite[Theorem 9]{Butler2}.

One may ask if the bound in Conjecture \ref{wilkie} holds for other o-minimal structures. For instance, we have the following natural generalization of Conjecture \ref{wilkie}, which appears in \cite{Jones}.
\begin{conjecture}
Let $f_{1},\ldots ,f_{r}$ be a Pfaffian chain and suppose that  $\tilde{\mathbb{R}}=(\mathbb{R},<,+,\cdot,f_{1},\ldots ,f_{r})$ is a model complete expansion of the real field. Suppose that $X\subseteq \mathbb{R}^{n}$ is a set definable in the o-minimal structure $\tilde{\mathbb{R}}$. For any number field $K\subseteq \mathbb{R}$, there exist positive constants $c_{1}=c_{1}(X,K)$, $c_{2}=c_{2}(X)$ such that
\begin{equation}
\left\vert X^{{\rm trans}}(K,N) \right\vert\leq c_{1}\left(\log(N)\right)^{c_{2}}
\end{equation}
for all $N>e$.
\label{wilkiepf}
\end{conjecture}
If the dimension of $X$ equals $1$, then Conjecture \ref{wilkie} is known to hold by work of Butler \cite{Butler} and Jones and Thomas \cite{Jones}. If $X$ has dimension greater than $1$, Conjecture \ref{wilkie} is known for the family of surfaces 
\begin{equation}
\left\{(x,y,z)\in (0,\infty)^{3}:(\log(x))^{a}(\log(x))^{b}(\log(z))^{c}=1\right\},\; (a,b,c)\in \mathbb{Q}^{3},
\label{family} 
\end{equation}
by work of Pila \cite{Pila4} when $(a,b,c)=(1,1,-1)$ and Butler \cite{Butler} in the general case. If $X\subseteq \mathbb{R}^{3}$ is definable in an o-minimal structure as in Conjecture \ref{wilkiepf}, then, under the assumption that $X$ has a mild parametrization (see \cite[$\S$ 2]{Pila4}), in \cite{Jones} Jones and Thomas prove that $X$ satisfies Conjecture \ref{wilkiepf}. Furthermore, in \cite[Proposition 2.3.6]{CluckersPila}, Cluckers, Pila and Wilkie prove a parametrization result and then use it to show  that if $\mathcal{X}=\{X_{t}:t\in T\}\subseteq T\times (-1,1)^{3}$ is a family of Pfaffian surfaces definable in the restricted analytic field $\mathbb{R}_{\text{an}}$, then Conjecture \ref{wilkiepf} holds uniformly on the fibres of $\mathcal{X}$, specifically, there are positive constants $c_{1}=c_{1}(\mathcal{X})$, $c_{2}=c_{2}(\mathcal{X})$ such that  
$$\left\vert X_{t}^{{\rm trans}}(\mathbb{Q},N)\right\vert \leq c_{1}\left(\log(N)\right)^{c_{2}}$$
for all $t\in T$ and for all $N>e$. 

More generally, in \cite{Binyamini} Binyamini and Novikov prove that if $X\subseteq \mathbb{R}^{n}$ is a set definable in the ``restricted'' o-minimal structure $\mathbb{R}^{{\rm RE}}:=(\mathbb{R},<,+,\cdot, \exp|_{[0,1]},\sin|_{[0,\pi]})$, then $X$ satisfies the bound in Conjecture \ref{wilkie} (in fact, Binyamini and Novikov prove a stronger bound, see \cite[Theorem 2]{Binyamini}). We remark that the approach of Binyamini and Novikov does not rely on mild parametrizations. 

The purpose of this article is to pose a strategy to prove Conjecture \ref{wilkiepf} which does not use mild parametrizations. More specifically, our result implies the following consequence:

\begin{theorem}
Let $X\subseteq \mathbb{R}^{n}$ be a set definable in $\mathbb{R}_{\exp}$ of dimension at most $2$. Let $K$ be a number field and let $\mathcal{O}_{K}$ be its ring of integers. Then
\begin{align*}
|X^{{\rm trans}}(\mathcal{O}_{K},N)|& =|\{{\bf x} =(x_{1},\ldots ,x_{n})\in X^{{\rm trans}}\cap \mathcal{O}_{K}^{n}:H(x_{i})\leq N\;\forall i\}|\\ & \leq c_{1}(X,K)(\log(N))^{c_{2}(X)},
\end{align*}
for some positive constants $c_{1}(X,K)$, $c_{2}(X)$ if and only if there exist positive constants $\alpha=\alpha_{1}(X,K)$, $\tau=\tau(X,K)$, $\kappa=\kappa(X)$, with $0\leq \kappa<n$, such that for all non-zero primes $\mathfrak{p}\subseteq \mathcal{O}_{K}$ of absolute norm $\mathcal{N}_{K}(\mathfrak{p})\geq \tau(\log(N))^{\frac{n}{n-\kappa}}$, it holds
\begin{align*}
|\{\left(x_{1}\emph{\text{ mod }}(\mathfrak{p}),\ldots ,x_{n}\emph{\text{mod}}(\mathfrak{p})\right) & :(x_{1},\ldots ,x_{n})\in X^{{\rm trans}}\cap \mathcal{O}_{K}^{n},H(x_{i})\leq N\;\forall i\}| \\ & \leq \alpha\mathcal{N}_{K}(\mathfrak{p})^{\kappa}
\end{align*}
for all $N>e$.
\label{app}
\end{theorem}
 
The proof of Theorem \ref{app} uses the polynomial method, but rather than using Bombieri-Pila determinant method as in \cite{Butler, Jones, Pila4}, we use a variant of Siegel's lemma. That a variant of Siegel's lemma could be applied to the problem of counting points in o-minimal structures is not new; Wilkie in \cite[Lecture 2, $\S$ 6.3, $\S$ 6.4]{JonesWilkie} gives a proof of Theorem \ref{pilawilkietheorem} similar to the one in \cite{Pila2}, using a lemma of Thue-Siegel instead of Bombieri-Pila determinant method. Both proofs, however, rely on the fact that any set definable in an o-minimal structure admits good parametrizations (see \cite[Theorem 2.3, Theorem 2.5]{Pila2}). The novelty in our approach resides that in order to apply Siegel's lemma, instead of proving that the sets possesses well-behaved parametrizations, we use that the integral points of a set $X$ definable in an o-minimal structure must occupy few residue classes modulo $p$ for many primes $p$. Indeed, let $X\subseteq \mathbb{R}^{n}$  be a set definable in $\mathbb{R}_{\exp}$. Given a prime $p\in \mathbb{Z}$, define
\begin{center}
$X_{p}:=\left \{(x_{1}.\ldots ,x_{n})\mod(p):(x_{1},\ldots ,x_{n})\in X\cap \mathbb{Z}^{n}\right \}$.
\end{center}
It is easy to show (see Section \ref{4}) that Conjecture \ref{wilkie} implies
\begin{equation}
|X^{\text{trans}}(\mathbb{Z},N)_{p}|\leq p^{\kappa}
\label{PilaWilkielocal}
\end{equation}
for all primes $c_{1}\log(N)^{\frac{n}{n-\kappa}}\leq p\leq 2c_{1}\log(N)^{\frac{n}{n-\kappa}}$, with $0\leq \kappa<n$. Conversely, suppose that for all $N>e$ there exist constants $c_{1}:=c_{1}(X)$, $\kappa=\kappa(X)$ such that the set $X$ verifies \eqref{PilaWilkielocal} for all primes $c_{1}\log(N)^{\frac{n}{n-\kappa}}\leq p\leq 2c_{1}\log(N)^{\frac{n}{n-\kappa}}$. Since $X^{\text{trans}}(\mathbb{Z},N)$ is, by definition, highly non-algebraic, if one could show that $X^{\text{trans}}(\mathbb{Z},N)$ possesses some sort of algebraic structure, one would expect that the set $X^{\text{trans}}(\mathbb{Z},N)$ is small.

In order to formalize the last claim, we recall the following result of Walsh that established a conjecture of Helfgott and Venkatesh in \cite{Helfgott} regarding the presence of algebraic structure in sets badly distributed at the level of residue classes. 

\begin{theorem}[{\cite[Theorem 1.3]{Walsh}}]
For every positive integer $d$, and real $0\leq \kappa<d$, there exists $\tau=\tau(d,\kappa)>0$ such that the following holds. Write $\mathcal{P}_{I}$ for the primes in the interval
\begin{center}
$I=\left[ \tau\left(\log(N)\right)^{\frac{d}{d-\kappa}},2\tau\left( \log(N) \right)^{\frac{d}{d-\kappa}} \right]$.
\end{center}
Then, for every $S\subseteq \{-N,\ldots ,-1,0,1,\ldots ,N\}^{d}$ occupying $\lesssim p^{\kappa}$ residue classes modulo $p$ for every $p\in \mathcal{P}_{I}$, and every $\varepsilon>0$, there exists some non-zero $P\in \mathbb{Z}[X_{1},\ldots ,X_{d}]$ of complexity $\lesssim_{\kappa,d,\varepsilon}(\log(N))^{\frac{\kappa}{d-\kappa}}$ vanishing on at least $(1-\varepsilon)|S|$ points of $S$
\label{WW}
\end{theorem}
Here, by a polynomial $P$ of complexity at most $C$ we mean that $P$ has degree at most $C$ and its coefficients are bounded by $N^{C}$. 

Now, we  may explain our strategy to prove Theorem \ref{app} in the case $K=\mathbb Q$. Let $X$ be the graph of the function $f$. By hypothesis, the set $X^{\text{trans}}(\mathbb{Z},N)$ occupies few residue classes modulo $p$ for all $p\in \mathcal{P}_{I}$. Using Theorem \ref{WW}  we conclude that some hypersurface $V(P)$ of degree $O((\log(N))^{c(X)})$ vanishes at a positive proportion of $X^{\text{trans}}(\mathbb{Z},N)$. To conclude that the set $(X\cap V(P))^{\text{trans}}(\mathbb{Z},T)$ is small, we use bounds for the complexity of the intersection of Pfaffian sets, as in \cite{Jones}.

We will see that the same strategy works for a number field $K\subseteq \mathbb{R}$, once we extend there Theorem \ref{WW}. In fact  in this article we are going to prove a generalization of Theorem \ref{WW} for global fields, replacing $\mathbb{Z}$ with the ring of integers $\mathcal{O}_{K}$, with $K$ a global field. A diophantine analogue of the set $\{-N,\ldots ,-1,0,1,\ldots,N\}$ is the subset of elements $x\in \mathcal{O}_{K}$ of affine height $H(x)$ at most $N$. We denote this set by $[N]_{\mathcal{O}_{K}}$. Then we have the following definition of complexity.
\begin{definition}[Complexity]
A non-zero polynomial $P\in \mathcal{O}_{K}[X_{1},\ldots ,X_{n}]$ has \emph{complexity at most $C$ in $[N]_{\mathcal{O}_{K}}^{n}$} if it has degree at most $C$ and its coefficients have affine height bounded by $N^{C}$.
\label{complexity}
\end{definition}
For a non-zero ideal $\mathfrak{a}\subseteq \mathcal{O}_{K}$, let $\mathcal{N}_{K}(\mathfrak{a})$ be the absolute norm of $\mathfrak{a}$, defined as the (finite) cardinal of the set $\mathcal{O}_{K}/\mathfrak{a}$. This allows us to generalize the notion of ill-distributed set at the level of residue classes, in a straightforward way. For a prime ideal $\mathfrak{p}\in \mathcal{O}_{K}$, and a set $X\subseteq \mathcal{O}_{K}^{n}$, we denote $X_{\mathfrak{p}}$ for the set of residue classes of $X$ modulo $\mathfrak{p}$:
\begin{equation}
X_{\mathfrak{p}}=\left \{(x_{1}.\ldots ,x_{n})\mod(\mathfrak{p}):(x_{1},\ldots ,x_{n})\in X\right \}.
\end{equation}
Following the same strategy of Walsh, we prove the next generalization of Theorem \ref{WW}  
\begin{theorem}
For all $n>0$, all real $0\leq \kappa<n$ and all global field $K$, there exists $\tau=\tau(n,\kappa,K)\geq 1$ such that the following holds. Denote $I$ for the interval $[\tau(\log(N))^{\frac{n}{n-\kappa}},2\tau(\log(N))^{\frac{n}{n-\kappa}}]$. Write $\mathcal{P}_{I,K}$ for the set of primes ideals $\mathfrak{p}\subseteq \mathcal{O}_{K}$ defined as
\begin{center}
$\mathcal{P}_{I,K}:=\begin{cases}\left\{\mathfrak{p}\subseteq \mathcal{O}_{K}:\mathcal{N}_{K}(\mathfrak{p})\in I\right\} \;\text{if}\;K\;\text{is a number field},\\ \left\{ \mathfrak{p}\subseteq \mathcal{O}_{K}:\mathcal{N}_{K}(\mathfrak{p})= \tau(\log(N))^{\frac{n}{n-\kappa}}\right\}\;\text{if}\;K\;\text{is a function field}.\end{cases}$
\end{center}
Then, for every $X\subseteq [N]_{\mathcal{O}_{K}}^{n}$ with $|X_{\mathfrak{p}}|\lesssim \mathcal{N}_{K}(\mathfrak{p})^{\kappa}$ for every prime $\mathfrak{p}\in \mathcal{P}_{I,K}$, and every $\varepsilon>0$, there exists some non-zero $P\in \mathcal{O}_{K}[X_{1},\ldots ,X_{n}]$, of complexity $\lesssim_{\kappa,n,\varepsilon,K}(\log(N))^{\frac{\kappa}{n-\kappa}}$ in $[N]_{\mathcal{O}_{K}}^{n}$ vanishing on at least $(1-\varepsilon)|X|$ points of $X$.  
\label{generalization}
\end{theorem}
In fact, we obtain a more general theorem than Theorem \ref{generalization}, in which the set $X$ lies in a projective variety. For the sake of simplicity, we defer the details to section $\S 3$, where this generalization is proved (see Theorem \ref{affinegeneralization}).

 \section{Heights in global fields}

\subsection{Absolute values}

The references for this section are the first two chapters of \cite{Bombieri}, and section B of \cite{Hindry} for the basic theory of heights, and chapter 5 of \cite{Rosen} and chapter 1 of \cite{Stick} for the theory of function fields. 

Throughout this paper, $k$ denotes either the field $\mathbb{Q}$ of rational numbers or the field $\mathbb{F}_{q}(T)$ of rational functions in one indeterminate over a finite field $\mathbb{F}_{q}$. We fix an algebraic closure $\overline{k}$ of $k$ and denote by $K\subseteq \overline{k}$ a global field, i.e. a finite separable extension of $k$.

Let us denote $M_{K}$ for the set of places $v$ of $K$. For each $v\in M_{K}$ let $K_{v}$ be the completion of $K$ with respect to $v$. If $\mathcal{O}_{v}$  is the valuation ring of $v$ in $K_{v}$, we denote $\mathfrak{m}_{v}$ for its maximal ideal. Following \cite{Bombieri}, we take normalized representatives $|\cdot |_{v}$ for the places $v\in M_{K}$. First, suppose that $K=k=\mathbb{Q}$.

\begin{itemize}
\item[(i)] If $v=\infty$, then $|\cdot|_{v}$ is the usual archimedean value of $k$;
\item[(ii)] If $v$ corresponds to a prime $p$, then $|\cdot|_{p}$ is the $p$-adic absolute value in $k$, with $|p|_{p}=p^{-1}$. 
\end{itemize}

Suppose now that $K=k=\mathbb{F}_{q}(T)$.
\begin{itemize}
\item[(i)] If $v$ corresponds to an irreducible polynomial $p\in \mathbb{F}_{q}[X]$, then $|\cdot |_{p}$ is the $p$-adic absolute value in $k$, with $|f|_{p}=q^{-\text{ord}_{p}(f)\deg(p)}$, $\text{ord}_{p}(f)$ being the order of $p$ in $f$;
\item[(ii)] If $v=\infty$ is the absolute value with $1/T\in \mathfrak{m}_{v}$, then $|\cdot|_{\infty}$ is the non-archimedean absolute value in $k$ with $|f|_{\infty}=q^{-\deg(f)}$, where $\deg(f)=\deg(h)-\deg(g)$ if $f=g/h$.
\end{itemize}

Now, for general $K$, let $w\in M_{K}$ be a place of $K$ which is over $v\in M_{k}$. We consider the normalized representative $||\cdot ||_{w}$ given by
\begin{equation}
||x||_{w}:=|N_{K_{w}/k_{v}}(x)|_{v}^{\frac{1}{[K:k]}}.
\label{placedefinition}
\end{equation}

The product formula is then
\begin{equation}
\displaystyle \prod_{w\in M_{K}}||x||_{w}=1
\label{product formula}
\end{equation}
for all $x\in K^{\times}$.

For a global field $K$, $M_{K,\infty}$ will be the set of places lying over the place $v=\infty\in M_{k}$. We have that $M_{K,\infty}$ has at most $[K:k]$ elements. The remaining places $M_{K,\text{fin}}:=M_{K}\backslash M_{K,\infty}$ are the finite places.

The ring of integers of $K$, which we will denote $\mathcal{O}_{K}$, is defined as the intersection of the valuation rings $\mathcal{O}_{v}$ for $v\in M_{K,\text{fin}}$

\begin{center}
$\mathcal{O}_{K}:=\displaystyle \bigcap_{v\in M_{K,\text{fin}}}\{x\in K:||x||_{v}\leq 1\}$.
\end{center}

Taking $K=k$, we have $\mathcal{O}_{k}=\mathbb{Z}$ or $\mathcal{O}_{k}=\mathbb{F}_{q}[T]$. In fact, $\mathcal{O}_{K}$ is the integral closure of $\mathcal{O}_{k}$ in $K$. A prime $\mathfrak{p}$ of $K$ is a non-zero prime ideal of $\mathcal{O}_{K}$ and it is in one-one correspondence with the maximal ideals $\mathfrak{m}_{v}$ with $v\in M_{k,\text{fin}}$. We have that the quotient field $\mathcal{O}_{K}/\mathfrak{p}$ is isomorphic to $\mathcal{O}_{v}/\mathfrak{m}_{v}$, where $v$ is the finite place that corresponds to $\mathfrak{p}$. In particular, this quotient is a finite field extending $\mathbb{F}_{q}$ and we denote its cardinal by $\mathcal{N}_{K}(\mathfrak{p})$; it is  the absolute norm of $\mathfrak{p}$. More generally, for any non-zero ideal $\mathfrak{a}\subsetneq \mathcal{O}_{K}$, we define $\mathcal{N}_{K}(\mathfrak{a})=|\mathcal{O}_{K}/\mathfrak{a}|$; this definition is multiplicative in the ideals. 

When $K$ is a function field, we have $\mathcal{N}_{K}(\mathfrak{p})=q^{[\mathcal{O}_{K}/\mathfrak{p}:\mathbb{F}_{q}]}$. The number $[\mathcal{O}_{K}/\mathfrak{p}:\mathbb{F}_{q}]$ is called the degree of $\mathfrak{p}$, and we denote it by $\deg(\mathfrak{p})$. Moreover, any place $w\in M_{K,\infty}$ is ultrametric, and the quotient field $\mathcal{O}_{w}/\mathfrak{m}_{w}$ is finite. We define the norm of $w$ and denote it by $\mathcal{N}_{K}(\mathfrak{m}_{w})$, as the cardinal of $\mathcal{O}_{w}/\mathfrak{m}_{w}$, and as before we have $\mathcal{N}_{K}(\mathfrak{m}_{w})=q^{[\mathcal{O}_{w}/\mathfrak{m}_{w}:\mathbb{F}_{q}]}$. The number $[\mathcal{O}_{w}/\mathfrak{m}_{w}:\mathbb{F}_{q}]$ is also called the degree of $w$, and we denote it by $\deg(w)$. 

\begin{remark}
For function fields defined over $\mathbb{F}_{q}$, there is another definition of a prime in $K$, which is more standard (see \cite{Rosen,Stick}): a prime in $K$ is a maximal ideal $\mathfrak{m}_{v}$ of a discrete valuation ring $\mathcal{O}_{v}$ with $\mathbb{F}_{q}\subsetneq \mathcal{O}_{v}\subsetneq K$. The definition we gave above coincides with this one in the primes of $M_{K}\backslash M_{K,\infty}$.  Also, we note that $\mathcal{O}_{K}$ is non-canonical: it depends on the transcendental element $T$ in the definition of $\mathcal{O}_{k}=\mathbb{F}_{q}[T]$. Instead, we could take a place $v\in M_{k}$ represented by the absolute value $|\cdot|_{p}$ corresponding to an irreducible polynomial $p(T)\in \mathbb{F}_{q}[T]$, and define $\mathcal{O}_{K}=\bigcap_{w\not | v}\mathcal{O}_{w}$. For instance, if $v$ corresponds to the irreducible polynomial $p(T)=T$, then $\mathcal{O}_{k}=\mathbb{F}_{q}[\frac{1}{T}]$ and $\mathcal{O}_{K}$ would be its integral closure over $K$.
\end{remark}

We now write the normalization \eqref{placedefinition} for the places $w\in M_{K,\text{fin}}$ in terms of valuations. We have that such $w$ corresponds to a prime ideal $\mathfrak{p}$ of $\mathcal{O}_{K}$, obtained as $\mathfrak{m}_{w}\cap \mathcal{O}_{K}$. We denote $\text{ord}_{\mathfrak{p}}$ for the corresponding normalized discrete valuation on $\mathfrak{p}$. Using \cite[Chapter 1, Proposition 2.5]{Lang} and following the remarks in \cite[Chapter 2, $\S 2$]{Lang} which also work for global fields which are function fields, we can write \eqref{placedefinition} as
\begin{equation}
||x||_{w}=\mathcal{N}_{K}(\mathfrak{p})^{-\frac{\text{ord}_{\mathfrak{p}}(x)}{[K:k]}}.
\label{placedefinition2}
\end{equation}
With \eqref{placedefinition2} we can express the norm of an element $x\in \mathcal{O}_{K}\backslash \{0\}$ in a convenient way. Indeed, the ideal $(x)$ factorizes as $\prod_{\mathfrak{p}\in S}\mathfrak{p}^{\text{ord}_{\mathfrak{p}}(x)}$.  Let $w_{\mathfrak{p}}$ be the corresponding place associated to $\mathfrak{p}$. Then

\begin{equation}
\mathcal{N}_{K}(x)^{\frac{1}{[K:k]}}=\prod_{\mathfrak{p}\in S}\mathcal{N}_{K}(\mathfrak{p})^{\frac{\text{ord}_{\mathfrak{p}}(x)}{[K:k]}}=\prod_{w\in S}||x^{-1}||_{w}.
\label{dedekind}
\end{equation}

\subsection{Heights} The usual projective height for any $(x_{1},\ldots ,x_{n})\in \overline{k}^{n}$ with coordinates in a global field is defined in the following way. If $K$ is a global field in which the coordinates of  ${\bf x}$ are defined, then 
\begin{center}
$H(x_{1},\ldots ,x_{n})=\displaystyle \prod_{v\in M_{K}}\max_{1\leq i\leq n}||x_{i}||_{v}$.
\end{center}
This definition is invariant under multiplication by non-zero scalars. Thus, given ${\bf x}\in \mathbb{P}^{n}(\overline{k})$ with coordinates $(x_{0}:x_{1}:\ldots :x_{n})$ with $x_{i}\in K$ for all $i$, $K$ a global field, we define $H({\bf x})=H(x_{0}:\ldots :x_{n})=H(x_{0},\ldots ,x_{n})$. If $x\in \overline{k}$ and lies in a global field, then $H(x)$ (the affine height of $x$) will always denote the projective height $H(1,x)$. Note that if $x\in \mathbb{Z}$ then $H(x)=|x|$, the absolute value of $x$, and if $x\in \mathbb{F}_{q}[T]$, then $H(x)=q^{\deg(x)}$, where $\deg(x)$ is the degree of $x$. In these two cases, $H(x)=\mathcal{N}_{K}(x)$. For $k=\mathbb{F}_{q}(T)$, let $K/k$ be a finite separable extension. Then
\begin{align}
H(x)=\prod_{v\in M_{K}}\max\{1,||x||_{v}\}=\prod_{v\in M_{K},||x||_{v}\geq 1}||x||_{v} & =\prod_{v\in M_{K},||x||_{v}\geq 1}q^{\frac{a_{v}}{[K:k]}}\label{altura para cuerpos funcionales}
\\ & =q^{\frac{a}{[K:k]}}, \;a\in \mathbb{Z}_{\geq 0}. \nonumber
\end{align}
In particular, if $K$ is function field, it is more natural to count points of height equal to a parameter $N=q^{\frac{a}{[K:k]}}$ for some positive integer $a$, instead of counting points of height bounded by a parameter $N$, as in the number field case. However, for this article it will be more convenient to consider the set of points of height bounded by $N$.

For our purposes, it will be necessary to understand how the affine height of a point behaves under the action of a polynomial. It is easy to show (see \cite[Proposition B.2.5. (a)]{Hindry}) that if $\phi:\mathbb{P}^{n}\rightarrow \mathbb{P}^{m}$ is a rational map of degree $D$ defined over $\overline{k}$, $\phi=(f_{0}:\ldots :f_{m})$ with $f_{i}$ homogeneous polynomials of degree $D$, and $Z\subseteq \mathbb{P}^{n}$ is the subset of common zeros of the $f_{i}$'s (so $\phi$ is defined on $\mathbb{P}^{n}\backslash Z$), then
\begin{equation}
H(\phi({\bf x}))\leq R H({\bf a})H({\bf x})^{D}
\label{polynomialheight}
\end{equation}
for all ${\bf x}\in \mathbb{P}^{n}(\overline{k})\backslash Z$ with coordinates lying in a global field, where $R$ is the maximum number of monomials appearing in any one of the $\phi_{i}$, and ${\bf a}$ is the projective point with coordinates the coefficients of all the $\phi_{i}$. It follows that the same upper bound \eqref{polynomialheight} holds for $H(1:P({\bf x}))$, where $P(T_{1},\ldots ,T_{n})\in K[T_{1},\ldots ,T_{n}]$ and ${\bf x}\in \overline{k}^{n}$, namely, if $P(T_{1},\ldots ,T_{n})=\sum_{(i_{1},\ldots ,i_{n})}c_{i_{1},\ldots ,i_{n}}T_{1}^{i_{1}}\cdots T_{n}^{i_{n}}$, ${\bf c}=(c_{i_{1},\ldots ,i_{n}})_{i_{1},\ldots ,i_{n}}$ and $R$ is the number of $n$-tuples $(i_{1},\ldots, i_{n})$ with $c_{i_{1},\ldots ,i_{n}}\neq 0$, we have
\begin{equation}
H(P({\bf x}))\leq RH(1:{\bf c})H(1:{\bf x})^{\deg(P)}.
\label{polynomialheight2}
\end{equation}

Given ${\bf x}=(x_{1},\cdots ,x_{n})\in K^{n}$, we have the bound
\begin{equation}
H(x_{1}:\ldots :x_{n})\leq H(1:x_{1}:\ldots :x_{n})\leq \max_{i}\{H(x_{i})\}^{n}.
\label{inequality of heights0}
\end{equation}
Also, if ${\bf x}=(x_{1},\ldots ,x_{n})\in \mathcal{O}_{K}^{n}$, we have the bound
\begin{equation}
H(x_{1}:\ldots :x_{n})\leq H(1:x_{1}:\ldots :x_{n})\leq \max_{i}\{H(x_{i})\}^{[K:k]}.
\label{inequality of heights}
\end{equation}
We will use the notation:
\begin{center}
$[N]_{\mathcal{O}_{K}}^{n}:=\{{\bf x}=(x_{1},\ldots ,x_{n})\in \mathcal{O}_{K}^{n}:\max_{i}\{H(x_{i})\}\leq N\}$,

$[N]_{K}^{n}:=\{{\bf x}=(x_{1},\ldots ,x_{n})\in K^{n}:\max_{i}\{H(x_{i})\}\leq N\}$,

$[N]_{\mathbb{A}^{n}(K)}:=\{{\bf x}=(x_{1},\ldots ,x_{n})\in K^{n}:H(1:x_{1}:\ldots :x_{n})\leq N\}$,

$[N]_{\mathbb{P}^{n}(K)}:=\{{\bf x}=(x_{0}:\ldots :x_{n})\in \mathbb{P}^{n}(K):H(x_{0}:\ldots :x_{n})\leq N\}$.
\end{center}

Finally, we will need to relate the height and the norm of a point in $x\in \mathcal{O}_{K}$. From equality \eqref{dedekind} and the fact that $H(x^{-1})=H(x)$ for all $x\neq 0$ it follows that

\begin{equation}
\mathcal{N}_{K}(x)\leq H(x)^{[K:k]}\;\text{ for all }x\in \mathcal{O}_{K}\backslash \{0\}.
\label{coefbound}
\end{equation}

\section{Ill-distributed sets in projective varieties}

If $Z$ is a geometrically irreducible projective hypersurface of degree $d$ and dimension $n$ over a finite field $\mathbb{F}_{q}$, the Lang-Weil estimate says that $|Z(\mathbb{F}_{q})|=(1+O_{d,n}(q^{-\frac{1}{2}}))q^{n}$. This implies that if $Z\subseteq \mathbb{P}^{n}(\mathbb{Q})$ is a geometrically irreducible projective hypersurface of degree $d$ over $\mathbb{Q}$, then its reduction modulo $p$ occupies $(1+O_{d,n}(p^{-\frac{1}{2}}))p^{n-1}$ residue classes modulo $p$ for almost all primes $p$. Since $|\mathbb{P}^{n}(\mathbb{F}_{p})|\sim_{n}p^{n}$, we can think of $Z$ as an ill-distributed set at the level of residue classes. Then, in the diophantine context, it is natural to work with sets $X\subseteq \{{\bf x}\in \mathbb{P}^{n}(\mathbb{Q}):H({\bf x})\leq N\}$, or more generally $X\subseteq \{{\bf x}\in Z(\mathbb{Q}):H({\bf x})\leq N\}$ where $Z$ is a projective variety defined over $\mathbb{Q}$, such that the image of $X$ in $\mathbb{P}^{n}(\mathbb{Z}/p\mathbb{Z})$ is small for many primes $p$. This is the approach that we will consider in this section.

From here on, $K$ will be a global field. We denote by $[N]_{\mathbb{P}^{n}(K)}$ the subset of ${\bf x}\in \mathbb{P}^{n}(K)$ with $H({\bf x})\leq N$. For any prime $\mathfrak{p}\subseteq \mathcal{O}_{K}$, consider the reduction modulo $\mathfrak{p}$ map $\pi_{\mathfrak{p}}:\mathbb{P}^{n}(K)\rightarrow \mathbb{P}^{n}(\mathcal{O}_{K}/\mathfrak{p})$, defined as follows. Given $x\in \mathbb{P}^{n}(K)$, choose coordinates $(x_{0}:\ldots :x_{n})$ such that for all $i$, $x_{i}=0$ or $\text{ord}_{\mathfrak{p}}(x_{i})\geq 0$ for all $i$, and there exists $0\leq i_{0}\leq n$ with $\text{ord}_{\mathfrak{p}}(x_{i})=0$. Such coordinates are unique modulo a scalar multiple $\lambda\in \mathcal{O}_{\mathfrak{p}}^{\times}$. Then $\tilde{{\bf x}}=(x_{0},\ldots ,x_{n})\mod \mathfrak{p}$ defines a non-zero point in $\mathbb{A}^{n+1}(\mathcal{O}_{K}/\mathfrak{p})$. We define $\pi_{\mathfrak{p}}({\bf x})$ as the point in $\mathbb{P}^{n}(\mathcal{O}_{K}/\mathfrak{p})$ defined by $\tilde{{\bf x}}$. Also, $\pi_{\mathfrak{p}}({\bf x})=\pi_{\mathfrak{p}}({\bf x}')$ will be denoted as ${\bf x}\equiv {\bf x}'\mod (\mathfrak{p})$. We note that if $P(T_{0},\ldots ,T_{n})$ is a non-zero homogeneous polynomial with coefficients in $\mathcal{O}_{K}$, then for any ${\bf x},{\bf x}'\in \mathbb{P}^{n}(K)$ with ${\bf x}\equiv {\bf x}'\mod(\mathfrak{p})$ if $P({\bf x})\equiv 0\mod (\mathfrak{p})$ then $P({\bf x}')\equiv 0\mod(\mathfrak{p})$.

For a non-zero prime ideal $\mathfrak{p}\in \mathcal{O}_{K}$, and a set $X\subseteq \mathbb{P}^{n}(K)$, we denote $X_{\mathfrak{p}}:=\pi_{\mathfrak{p}}(X)$. We have $|X_{\mathfrak{p}}|\leq |\mathbb{P}^{n}(\mathcal{O}_{K}/\mathfrak{p})|\lesssim_{n}\mathcal{N}_{K}(\mathfrak{p})^{n}$.

Consider $Z\subseteq \mathbb{P}^{m}(\overline{k})$ a projective variety defined over a global field $K$, with homogeneous ideal $\mathfrak{I}\subseteq \mathcal{O}_{K}[T_{0},\ldots ,T_{m}]$. defined by homogeneous polynomials  $f_{1},\ldots ,f_{l}\in \mathcal{O}_{K}[T_{0},\ldots ,T_{m}]$, so $Z=V(f_{1},\ldots ,f_{l})$. Call $\dim(Z)$ the dimension of $Z$. Let $M>0$ be a real number such that $l<M$ and $\dim(Z)<M$ and $\deg(f_{i})<M$ for all $i=1,\ldots ,l$. Let us suppose that $Z$ is geometrically irreducible. Then, Bertini-Noether theorem (\cite[Proposition 10.4.2 and Corollary 10.4.3 (a)]{Fried}) tells us that for all but finitely many primes $\mathfrak{p}$ of $K$ the reduction $Z_{\mathfrak{p}}:=V(\tilde{f_{1}},\ldots \tilde{f_{l}})$, where $\tilde{f_{i}}$ denotes the reduction modulo $\mathfrak{p}$ of $f_{i}$, remains geometrically irreducible and $\dim(Z_{\mathfrak{p}})=\dim(Z)<M$. 

Because of the Lang-Weil estimate we have that the reduction $Z_{\mathfrak{p}}$ has $(1+O_{M}(\mathcal{N}_{K}(\mathfrak{p})^{-1/2}))\mathcal{N}_{K}(\mathfrak{p})^{\dim(Z)}$ residue classes for almost every prime $\mathfrak{p}$.  In what follows we will show that an ill-distributed set in $Z$ has some sort of algebraic structure, with respect to the variety $Z$. For this, we adapt the notion of algebraic structure for subsets in $Z(K,N):=Z(K)\cap [N]_{\mathbb{P}^{m}(K)}$. This requires to extend the definition of complexity of a polynomial. Recall that we defined a non-zero polynomial $P\in \mathcal{O}_{K}[X_{1},\ldots ,X_{n}]$ has complexity at most $C$ in $[N]_{\mathcal{O}_{K}}^{n}$ if it has degree at most $C$ and its coefficients have affine height at most $N^{C}$.

\begin{definition}
Let $Z\subseteq \mathbb{P}^{m}(\overline{k})$ be a projective variety defined over $K$. We say that a homogeneous polynomial $P\in \mathcal{O}_{K}[T_{0},\ldots ,T_{m}]$ has \emph{complexity at most $C$ in $Z(K,N)$} if it does not vanish at $Z$, and there exist a non-zero homogeneous polynomial $Q\in \mathcal{O}_{K}[Y_{0},\ldots ,Y_{\dim(Z)}]$ of complexity at most $C$ in $[N]_{\mathcal{O}_{K}}^{\dim(Z)+1}$ and a linear change of variables $Y_{i}=L_{i}(T_{0},\ldots ,T_{m})$, $1\leq i\leq \dim(Z)$, with the height of the coefficients of the $L_{i}$ bounded by a constant dependent only on $Z$, such that it holds $P(T_{0},\ldots ,T_{m})=Q(L_{0}(T_{0},\ldots ,T_{m}),\ldots ,L_{\dim(Z)}(T_{0},\ldots ,T_{m}))$.
\end{definition}

Now we state the main result of this section.

\begin{theorem}
Let $Z\subseteq \mathbb{P}^{m}(\overline{k})$ be a projective variety defined over a global field $K$, and call $n=\dim(Z)$. Suppose that $Z$ is geometrically irreducible. For all real $0\leq \kappa<n$, there exists $\tau=\tau(n,\kappa,K,Z)\geq 1$ such that the following holds. Denote $I$ for the interval $[\tau(\log(N))^{\frac{n}{n-\kappa}},2\tau(\log(N))^{\frac{n}{n-\kappa}}]$. Write $\mathcal{P}_{I,K}$ for the set of prime ideals $\mathfrak{p}\subseteq \mathcal{O}_{K}$ defined as
\begin{center}
$\mathcal{P}_{I,K}:=\begin{cases}\left\{\mathfrak{p}\subseteq \mathcal{O}_{K}:\mathcal{N}_{K}(\mathfrak{p})\in I\right\}\;\text{if}\;K\;\text{is a number field},\\ \left\{ \mathfrak{p}\subseteq \mathcal{O}_{K}:\mathcal{N}_{K}(\mathfrak{p})= \tau(\log(N))^{\frac{n}{n-\kappa}}\right\}\;\text{if}\;K\;\text{is a function field}.\end{cases}$
\end{center}
Then, for every $X\subseteq Z(K,N)$ with $|X_{\mathfrak{p}}|\lesssim \mathcal{N}_{K}(\mathfrak{p})^{\kappa}$ for every prime $\mathfrak{p}\in \mathcal{P}_{I,K}$, and every $\varepsilon>0$, there exists some homogeneous polynomial $P\in \mathcal{O}_{K}[T_{0},\ldots ,T_{m}]$ of complexity $\lesssim_{\kappa,n,m,\varepsilon,K,Z}(\log(N))^{\frac{\kappa}{n-\kappa}}$ in $Z(K,N)$ and vanishing on at least $(1-\varepsilon)|X|$ points of $X$.  
\label{affinegeneralization}
\end{theorem}

Theorem \ref{WW} of Walsh \cite{Walsh} is a consequence of Theorem \ref{affinegeneralization} for $Z=\mathbb{P}^{m}(\overline{k})$ and $K=\mathbb{Q}$, while Theorem \ref{generalization} is the case where $Z=\mathbb{P}^{m}(\overline{k})$ and $K$ a global field. In fact, the strategy of the proof of Theorem \ref{affinegeneralization} is the same as the one given in \cite{Walsh}. Namely, for any set $X$ as in Theorem \ref{affinegeneralization} we will construct a small dense set $\mathcal{C}$ in $X$; by this we mean a set of small size, and such that polynomials of low complexity that vanish in $\mathcal{C}$ also vanish at a fixed positive proportion of $X$.

To find the dense set $\mathcal{C}$, we introduce the quantity
\begin{equation}
r=\eta \left(\log(N)\right)^{\frac{n\kappa}{n-\kappa}},
\label{r}
\end{equation}
where $\eta\geq 1$ is a constant that we shall choose later. .

\begin{proposition}
Let $X\subseteq Z(K,N)$ with $|X_{\mathfrak{p}}|\leq \alpha\mathcal{N}_{K}(\mathfrak{p})^{\kappa}$ for all prime ideals $\mathfrak{p}\in \mathcal{P}_{I,K}$. There exist a positive constant $C_{1}=C_{1}(\kappa)$, sets $\mathcal{C},X'\subseteq X$ of size $|\mathcal{C}|\leq r,|X'|\gtrsim |X|$, such that if $\eta\geq C_{1}\alpha\tau^{\kappa}$, then for every ${\bf x}\in X'$ we have
\begin{equation}
\displaystyle \sum_{\mathfrak{p}\in \mathcal{P}_{I,K}}1_{\exists c\in \mathcal{C}:{\bf x}\equiv c\mod(\mathfrak{p})}\log(\mathcal{N}_{K}(\mathfrak{p}))\gtrsim_{K} |I|,
\label{promedio}
\end{equation}
where $|I|=\tau(\log(N))^{\frac{n}{n-\kappa}}$.
\label{characteristicproposition2}
\end{proposition}
Note that if $K$ is a number field, then $|I|$ is the length of the interval $I$. 
\begin{proof}
The proof is exactly the same as Proposition 3.1 of \cite{Walsh}. We include it for the sake of completeness. Call $({\bf x},L)\in X\times X^{r}$ a good tuple modulo $\mathfrak{p}$ if there exists a coordinate of $L$ such that its reduction modulo $\mathfrak{p}$ coincides with the reduction modulo $\mathfrak{p}$ of ${\bf x}$. Let us denote $X_{\text{good},\mathfrak{p}}$ for the set of good tuples modulo $\mathfrak{p}$. Our set $\mathcal{C}$ will be constructed as the set of coordinates of an $L\in X^{r}$ such that $({\bf x},L)\in X_{\text{good},\mathfrak{p}}$ for many ${\bf x}\in X$ and many primes $\mathfrak{p}\in \mathcal{P}_{I,K}$. In order to prove this, first we prove that for a fixed prime $\mathfrak{p}\in \mathcal{P}_{I,K}$ the set $X_{\text{good},\mathfrak{p}}$ is big.

For any residue class ${\bf a}$ in $Z_{\mathfrak{p}}$, let us denote $X_{{\bf a}}$ the probability of ${\bf x}\in X$ such that ${\bf x}\equiv {\bf a}\mod(\mathfrak{p})$. To find many $({\bf x},L)$ which are good modulo $\mathfrak{p}$ it is enough to show that the probability of a tuple $({\bf x},L)$ not being good modulo $\mathfrak{p}$ is small. In other words, it is enough to give an upper bound $c<1$ for
\begin{equation}
\displaystyle \sum_{{\bf a}\in Z_{\mathfrak{p}}}X_{{\bf a}}(1-X_{{\bf a}})^{r}.
\label{proba}
\end{equation} 
If we sum over the ${\bf a}$'s  such that $X_{{\bf a}}>1/r$, then we get the upper bound $(1-1/r)^{r}$. Since this quantity approaches $e^{-1}$ as $r\rightarrow +\infty$, if $N$ is large enough (depending on $n$ and $\kappa$), we have
\begin{equation}
\displaystyle \sum_{{\bf a}\in Z_{\mathfrak{p}}:X_{{\bf a}}>1/r}X_{{\bf a}}(1-X_{{\bf a}})^{r}\leq c_{1}<1
\label{proba1}
\end{equation}
for some positive constant $c_{1}$. Now, if we sum over the ${\bf a}$'s such that $X_{{\bf a}}\leq 1/r$, then, taking into account that $X_{\mathfrak{p}}$ has at most $\alpha\mathcal{N}_{K}(\mathfrak{p})^{\kappa}$ elements for all $\mathfrak{p}\in \mathcal{P}_{I,K}$, we have
\begin{align}
\displaystyle \sum_{{\bf a}\in Z_{\mathfrak{p}}:0<X_{{\bf a}}\leq 1/r}X_{{\bf a}}(1-X_{{\bf a}})^{r} & \leq \displaystyle \sum_{{\bf a}\in Z_{\mathfrak{p}},0<X_{{\bf a}}\leq 1/r}\dfrac{1}{r}\cdot 1\leq \dfrac{\alpha}{r}\mathcal{N}_{K}(\mathfrak{p})^{\kappa} \label{proba2}\\ & \leq \dfrac{\alpha \left(2\tau(\log(N))^{\frac{n}{n-\kappa}}\right)^{\kappa}}{\eta(\log(N))^{\frac{n\kappa}{n-\kappa}}}\leq \dfrac{2^{\kappa}\alpha\tau^{\kappa}}{\eta}\leq \dfrac{1-c_{1}}{2},
\end{align}
where the last inequality can be achieved if we impose the condition 
\begin{equation}
\eta\geq C_{1}\alpha\tau^{\kappa}
\label{condicionA}
\end{equation}
for some explicit constant $C_{1}$, depending on $\kappa$. Combining \eqref{proba1} and \eqref{proba2} we get the upper bound $c=:c_{1}+\frac{1-c_{1}}{2}<1$ for \eqref{proba}, so there exists at least $(1-c)|X|^{r+1}$ tuples $({\bf x},L)\in X\times X^{r}$ which are good modulo $\mathfrak{p}$. In other words, $|X_{\text{good},\mathfrak{p}}|\geq (1-c)|X|^{r+1}$. Note that the constant $c$ is effective, and independent of $\mathfrak{p}$. 

From the fact that $|X_{\text{good},\mathfrak{p}}|\geq (1-c)|X|^{r+1}$, it follows that: 
\begin{fact}
For every prime ideal $\mathfrak{p}\in \mathcal{P}_{I,K}$ there exist absolute constants $c_{1}$ and $c_{2}$, both independent of $\mathfrak{p}$, such that for at least $c_{1}|X|^{r}$ choices of $L\in X^{r}$, there are at least $c_{2}|X|$ elements ${\bf x}\in X$ for which $({\bf x},L)\in X_{{\rm good},\mathfrak{p}}$.
\label{fact1}
\end{fact} 
Indeed, suppose that this does fail. Then, for some $\mathfrak{p}\in \mathcal{P}_{I,K}$ and for all positive constants $c_{1},c_{2}$, we have at most $c_{1}|X|^{r}$ choices for $L\in X^{r}$ such that there exist at least $c_{2}|X|$ elements of ${\bf x}\in X$ for which $({\bf x},L)\in X_{\text{good},\mathfrak{p}}$. Call $\mathcal{L}$ the set of $L\in X^{r}$ such that $({\bf x},L)\in X_{\text{good},\mathfrak{p}}$ for at least $c_{2}|X|$ elements of ${\bf x}\in X$. Then $\mathcal{L}$ has at most $c_{1}|X|^{r}$ elements. Recalling that we already proved the lower bound $|X_{\text{good},\mathfrak{p}}|\geq (1-c)|X|^{r+1}$, we have:
\begin{align*}
(1-c)|X|^{r+1} & \leq |X_{\text{good},\mathfrak{p}}|\\ & \leq |\{(\boldsymbol x,L)\in X_{\text{good},\mathfrak{p}}:L\in \mathcal{L}\}|+|\{(\boldsymbol x,L)\in X_{\text{good},\mathfrak{p}}:L\notin \mathcal{L}\}|\\ & \leq c_{1}|X|^{r+1}+(1-c_{1})c_{2}|X|^{r+1}.
\end{align*}
Taking $c_{1}$ and $c_{2}$ sufficiently small enough we arrive to a contradiction.

We say that an element $L\in X^{r}$ is good modulo $\mathfrak{p}$ if $({\bf x},L)$ is good modulo $\mathfrak{p}$ for at least $c_{2}|X|$ elements ${\bf x}\in X$. Let us denote $\mathcal{L}_{\mathfrak{p}}$ for the set of such $L$'s. Fact \ref{fact1} implies that for every prime ideal $\mathfrak{p}\in \mathcal{P}_{I,K}$ we have $|\mathcal{L}_{\mathfrak{p}}|\geq c_{1}|X|^{r}$, therefore we have
\begin{center}
$\displaystyle \sum_{\mathfrak{p}\in \mathcal{P}_{I,K}}|\mathcal{L}_{\mathfrak{p}}|\geq c_{1}|X|^{r}|\mathcal{P}_{I,K}|$.
\end{center}
It follows that there exists some $L'\in X^{r}$ such that $L'\in \mathcal{L}_{\mathfrak{p}}$ for at least $c_{1}|\mathcal{P}_{I,K}|$ prime ideals in $\mathcal{P}_{I,K}$. By construction, we have
\begin{equation}
\displaystyle\sum_{{\bf x}\in X}\left\vert \left\{ \mathfrak{p}\in \mathcal{P}_{I,K}:({\bf x},L')\in X_{\text{good},\mathfrak{p}} \right\} \right\vert\geq c_{1}c_{2}|X||\mathcal{P}_{I,K}|.
\label{eqq}
\end{equation}
We conclude that there exist positive constants $c_{3},c_{4}$ and a subset $X'\subseteq X$ of size $|X'|\geq c_{3}|X|$, such that for every ${\bf x}\in X'$ there are at least $c_{4}|\mathcal{P}_{I,K}|$ prime ideals $\mathfrak{p}\in \mathcal{P}_{I,K}$ for which $({\bf x},L')\in X_{\text{good},\mathfrak{p}}$. 

Take $\mathcal{C}\subseteq X$ to be the set of coordinates of $L'$, so $|\mathcal{C}|$ has at most $r$ elements.  Since $\mathcal{N}_{K}(\mathfrak{p})\geq |I|=\tau(\log(N))^{\frac{n}{n-\kappa}}$, we have that for every ${\bf x}\in X'$ it must be
\begin{equation}
\displaystyle \sum_{\mathfrak{p}\in \mathcal{P}_{I,K}}1_{\exists {\bf c}\in \mathcal{C}:{\bf x}\equiv {\bf c}\mod(\mathfrak{p})}\log(\mathcal{N}_{K}(\mathfrak{p}))\geq c_{4}|\mathcal{P}_{I,K}|\log(|I|).
\label{eqq1}
\end{equation}
If $K$ is a number field, we use the Landau Ideal Theorem to obtain $|\mathcal{P}_{I,K}|\sim_{K}\dfrac{|I|}{\log(|I|)}$. Replacing this bound in \eqref{eqq1} we conclude Proposition \ref{characteristicproposition2}. Suppose now that $K$ is a function field over $\mathbb{F}_{q}$. If $\pi_{K}(n)$ denotes the primes of $K$ of degree $n$, then the Riemann Hypothesis for curves over finite fields implies (see \cite[Theorem 5.12]{Rosen})
\begin{equation}
\pi_{K}(n)=\dfrac{q^{n}}{n}+O_{g}\left(\dfrac{q^{\frac{n}{2}}}{n}\right),
\label{rhforfunctionfields}
\end{equation}
where $g$ is the genus of $K$. Now, there may be primes lying at infinite that are being counted in $\pi_{K}(n)$, but since the degree of these primes is bounded by $[K:k]$ (use \cite[Proposition 1.1.15]{Stick} and the fact that any prime at infinite contains $1/T$), taking $n>[K:k]$ the number $\pi_{K}(n)$ counts only prime ideals of $\mathcal{O}_{K}$ of degree $n$. Recalling that $\mathcal{P}_{I,K}$ consists of primes of degree $H=\log_{q}(\tau \log(N)^{\frac{n}{n-\kappa}})=\log_{q}(|I|)$, we have
\begin{center}
$|\mathcal{P}_{I,K}|=\pi_{K}(H)=\dfrac{|I|}{\log_{q}(|I|)}+O_{g}\left(\dfrac{|I|^{\frac{1}{2}}}{\log_{q}(|I|)}\right)\gtrsim_{g,q}\dfrac{|I|}{\log|I|}$.
\end{center}
Replacing in \eqref{eqq1} we deduce Proposition \ref{characteristicproposition2}.
\end{proof}

\begin{remark}
The constant $C_{1}$ in Proposition \ref{characteristicproposition2} is effective. The proof shows that if we write $|X'|=\delta |X|$, then $\delta\geq c_{3}$ and $c_{3}$ is  an effective absolute constant. 

The implicit constant in \eqref{promedio} is effective if $K$ is a function field, since the implicit constant in the Riemann Hypothesis \eqref{rhforfunctionfields} is effective. If $K$ is a number field, this constant can be made explicit using an effective version of Landau's Ideal Theorem.\end{remark}

Having constructed the sets $\mathcal{C}$ and $X'$ of Proposition \ref{characteristicproposition2}, the next step is to construct a non-zero homogeneous polynomial $P\in \mathcal{O}_{K}[T_{0},\ldots ,T_{m}]$ of low complexity that vanishes at $\mathcal{C}$ and it is non-zero at $Z$. After this is done, we will show that such polynomial also vanishes at $X'$. Since $|X'|=\delta|X|$ for some $\delta>0$, this will allow us to conclude that $P$ vanishes on at least $\delta|X|$ points on $X$, concluding Theorem \ref{affinegeneralization} for $\varepsilon=\delta$. Theorem \ref{affinegeneralization} then follows upon $O_{\varepsilon}(1)$ iterations of this result.

To construct a polynomial of low complexity vanishing at $\mathcal{C}$ we will use the following version of Siegel's lemma, that includes both the number field and function field cases. We note that for number fields, we could use the result of Bombieri-Vaaler \cite[Corollary 11]{Bombieri2}, or even a more elementary result as \cite[Corollary 2.9.2.]{Bombieri}. For a lack of reference for the function field case, we provide a proof valid for both cases.

\begin{lemma}
Let $K$ be a global field with $[K:k]=d$. Let $(a_{ij})_{i,j}$, $1\leq i\leq s$, $1\leq j\leq t$ be elements of $\mathcal{O}_{K}$ with $H(a_{ij})\leq C$ for all $i,j$. Let us suppose that $t>2d^{2}s$. Then, there exists ${\bf c}=(c_{1},\ldots ,c_{t})\in \mathcal{O}_{K}^{t}\backslash\{0\}$, such that 
\begin{equation}
H(1:{\bf c})\lesssim_{K} (tC)^{\frac{4d^{2}s}{t-2d^{2}s}}
\label{con1}
\end{equation}
and
\begin{equation}
\displaystyle \sum_{j=1}^{t}c_{j}a_{ij}=0\; \text{for all}\;1\leq i\leq s.
\label{con2}
\end{equation}
\label{Siegel}
\end{lemma}
\begin{proof}
Let $h\geq 1$ be a parameter to be chosen later. Let $A_{\mathcal{O}_{K}}(h)$ and $A_{K}(h)$ be, respectively, the number of points in $\mathcal{O}_{K}$ and $K$ of height at most $h$. There are $(A_{\mathcal{O}_{K}}(h))^{t}$ $t$-tuples $(c_{1},\ldots ,c_{t})$ with $H(c_{i})\leq h$ for all $1\leq i\leq t$. For any such choice, \eqref{polynomialheight}  and \eqref{inequality of heights} imply
\begin{equation}
H\left( \displaystyle \sum_{j=1}^{t}c_{j}a_{ij} \right)\leq tH(1:{\bf c})H(1:a_{i1}:\ldots :a_{it})\leq t(hC)^{d}.
\end{equation}
Then, there are $(A_{\mathcal{O}_{K}}(t(hC)^{d}))^{s}$ possible configurations for all the sums $\sum_{j=1}^{t}c_{j}a_{ij}$. If 
\begin{equation}
(A_{\mathcal{O}_{K}}(h))^{t}>(A_{\mathcal{O}_{K}}(t(hC)^{d}))^{s},
\label{eq0}
\end{equation}
then there exists two tuples ${\bf c}_{1},{\bf c}_{2}\in \mathcal{O}_{K}^{t}\backslash\{0\}$, ${\bf c}_{1}=(c^{(1)}_{1},\ldots ,c^{(1)}_{t}), {\bf c}_{2}=(c^{(2)}_{1},\ldots ,c^{(2)}_{t})$, ${\bf c}_{1}\neq {\bf c}_{2}$ such that
\begin{equation}
\displaystyle \sum_{j=1}^{t}c^{(1)}_{j}a_{ij}=\displaystyle \sum_{j=1}^{t}c^{(2)}_{j}a_{ij}\;\text{for all}\;1\leq i\leq s.
\label{eq1}
\end{equation}
Note that ${\bf c}={\bf c}_{1}-{\bf c}_{2}$ satisfies \eqref{con2}. Since $H(x+y)\leq 2H(x)H(y)$ for all $x,y\in \overline{k}$, we have $H(c_{j}^{(1)}-c_{j}^{(2)})\leq 2h^{2}$ for all $1\leq j\leq t$. Then inequality  \eqref{inequality of heights} gives the bound $H(1:{\bf c})\leq (2h^{2})^{d}$. We will see that there exists an adequate $h$ such that \eqref{eq0} holds, and that for this choice of $h$, ${\bf c}$ satisfies \eqref{con1}. Note that
\begin{equation}
(A_{\mathcal{O}_{K}}(h))^{t}\geq (A_{\mathcal{O}_{k}}(h))^{t}\text{ and }\;(A_{K}(t(hC)^{d}))^{s}>(A_{\mathcal{O}_{K}}(t(hC)^{d}))^{s}.
\label{eq2}
\end{equation}
It is easy to see that $A_{\mathcal{O}_{k}}(h)\sim_{k} h$. Now, if $K$ is a number field, Schanuel's theorem \cite{Schanuel} says that $A_{K}(h)\sim_{K} h^{2d}$. If $K$ is a function field over $\mathbb{F}_{q}$, recall that the height of a point $x\in K$ is of the form $H(x)=q^{\frac{l}{d}}$ for some positive integer $l$. By \cite[Corollary 4.3]{Wan}, the number of points $x\in K$ with height $H(x)=q^{\frac{l}{d}}$ is $\sim_{K} q^{2l}$. This implies that $A_{K}(h)\sim_{K} h^{2d}$. Thus, choosing $h$ such that
\begin{equation}
h^{t}\gtrsim_{K} (t(hC)^{d})^{2ds},
\label{eq3}
\end{equation}
namely
\begin{equation}
h\sim_{K}(tC^{d})^{\frac{2ds}{t-2d^{2}s}},
\end{equation}
then $(A_{\mathcal{O}_{k}}(h))^{t}\geq (A_{K}(t(hC)^{d}))^{s}$. This, together with \eqref{eq2}, imply \eqref{eq0} and \eqref{con1}.
\end{proof}

For us to use Lemma \ref{Siegel} we need to have ``small coordinates'' for the points of the set $\mathcal{C}$. This is possible because of the following lemma.

\begin{lemma}[{See \cite[$\S$ 13.4]{Serre}, \cite[Proposition 2.1]{P2}}]
Let $K$ be a global field and let $n\geq 1$ be an integer. There exists a positive constant $c=c(K,n)$ such that for all ${\bf y}\in \mathbb{P}^{n}(K)$ there are coordinates $(y_{0},\ldots ,y_{n})\in \mathcal{O}_{K}^{n+1}$ such that
$$H(1:y_{0}:\ldots :y_{n})\leq cH({\bf y}).$$
\label{lemita}
Moreover, the constant $c$ is effectively computable. Thus, a subset $S\subseteq [N]_{\mathbb{P}^{n}(K)}$ can be lifted to a subset $\overline{S}\subseteq [cN]_{\mathbb{A}^{n+1}(\mathcal{O}_{K})}$.
\end{lemma}

We can now start the proof of Theorem \ref{affinegeneralization}  
 
\begin{proof}[Proof of Theorem \ref{affinegeneralization}]
Let $\mathcal{C}\subseteq X$ and $X'$ be the sets of Proposition \ref{characteristicproposition2}. As we have already explained, the first step is to construct a polynomial of low complexity that vanishes at $\mathcal{C}$, by means of Lemma \ref{Siegel}. If $Z=\mathbb{P}^{m}(\overline{k})$, to find a non-zero homogeneous $P\in \mathcal{O}_{K}[T_{0},\ldots ,T_{m}]$ of degree $D$, that vanishes at $\mathcal{C}$ amounts to solving a linear system of equations $A\cdot {\bf c}=0$. Hence, we can use Lemma \ref{Siegel} to find a non-zero polynomial of degree $D$ such that its coefficients have small height. However, in the case $Z\subsetneq \mathbb{P}^{m}(\overline{k})$, if we apply Lemma \ref{Siegel} directly, we would find a non-zero homogeneous polynomial $P\in \mathcal{O}_{K}[T_{0},\ldots ,T_{m}]$ of low complexity in $[N]_{\mathcal{O}_{K}}^{m+1}$, vanishing at $\mathcal{C}$, but it may happen that $P$ is identically zero at $Z$. To avoid this difficulty, we will find new variables $Y_{0},\ldots ,Y_{n}$ which are algebraically independent over $\overline{k}$, and then apply Lemma \ref{Siegel} to find a polynomial in this new set of variables. 

In order to find the new variables, we consider a dominant morphism ${\bf F}=(F_{0}:\ldots :F_{n}):Z\rightarrow \mathbb{P}^{n}$, with $F_{i}\in \mathcal{O}_{K}[T_{0},\ldots, T_{m}]$ for all $i$. Furthermore, using Noether's normalization and the facts that $k$ is infinite and $Z$ geometrically irreducible, we take ${\bf F}$ to be a finite morphism where each $F_{i}$ is a linear form with coefficients of height bounded by $\lesssim_{Z}1$. For each $i$, denote $Y_{i}:=F_{i}(T_{0},\ldots ,T_{m})$. Let $\tilde{\mathcal{C}}\subseteq \mathbb{P}^{n}(K)$ denote the image of $\mathcal{C}\subseteq \mathbb{P}^{m}(K)$ under ${\bf F}$. Note that $|\tilde{\mathcal{C}}|\leq |\mathcal{C}|$. If ${\bf x}\in [N]_{\mathbb{P}^{m}(K)}$, inequality \eqref{polynomialheight} gives

\begin{equation}
H(F_{0}({\bf x}),\ldots ,F_{n}({\bf x}))\leq c(Z)H({\bf x})\leq c(Z)N,
\label{polylineal}
\end{equation}
where $c(Z)$ is a constant dependent only on $Z$. Denote $N_{1}=c(Z)N$. Inequality \eqref{polylineal} means that $\tilde{\mathcal{C}}\subseteq [N_{1}]_{\mathbb{P}^{n}(K)}$. Now, let $D>0$ be an integer that we shall choose later, and let $\mathcal{R}$ the set of monomials of degree $D$ in $Y_{0},\ldots ,Y_{n}$. Then $R:=|\mathcal{R}|={D+n\choose n}$. This is also the number of $n$-tuples $(i_{0},\ldots ,i_{n})\in \mathbb{N}_{0}^{n}$ with $i_{0}+\cdots +i_{n}=D$. If ${\bf y}\in \tilde{\mathcal{C}}$, choose coordinates $(y_{0}:\ldots : y_{n})$ as in Lemma \ref{lemita}. Since ${\bf y}\in [N_{1}]_{\mathbb{P}^{n}(K)}$, from Lemma \ref{lemita} we conclude that 
\begin{center}
$H(1:y_{0}:\ldots :y_{n})\leq cN_{1}.$ 
\end{center}
For such coordinates of ${\bf y}$, and $I=(i_{0},\ldots ,i_{n})\in \mathcal{R}$, let us denote ${\bf y}^{I}:=y_{0}^{i_{0}}\cdots y_{n}^{i_{n}}$. Then $A:=({\bf y}^{I})_{{\bf y}\in \tilde{\mathcal{C}},I\in \mathcal{R}}$ is a $|\tilde{\mathcal{C}}|\times R$-matrix with entries in $\mathcal{O}_{K}$. Also, because of \eqref{polynomialheight2} and the choice of coordinates of ${\bf y}$, we have
\begin{center}
$H({\bf y}^{I})\leq H(1:y_{0}:\ldots :y_{n})^{D}\leq (cN_{1})^{D}$.
\end{center}
Now, choose $D$ such that the next inequalities hold:
\begin{equation}
((2d^{2}+1)n!|\tilde{\mathcal{C}}|)^{\frac{1}{n}}\leq ((2d^{2}+1)n!|\mathcal{C}|)^{\frac{1}{n}}\leq ((2d^{2}+1)n!r)^{\frac{1}{n}}<D\lesssim_{n,d}r^{\frac{1}{n}}.
\label{condition1}
\end{equation}
The inequality $D>((2d^{2}+1)n!r)^{\frac{1}{n}}$ gives $R={D+n\choose n}>(2d^{2}+1)r$. Moreover, this implies
\begin{equation}
\dfrac{r}{{D+n\choose n}-2d^{2}r}\leq 1.
\label{condition2}
\end{equation}
The inequality $D\lesssim_{n,d}r^{\frac{1}{n}}$ gives 
\begin{equation}
R={D+n\choose n}\lesssim_{n,d} r 
\label{upper bound for R}
\end{equation}
for $r$ large enough, and thus, for $N$ large enough. Since $|\tilde{\mathcal{C}}|\leq |\mathcal{C}|\leq r<R$, the $K$-subspace of solutions of the equation $A\cdot {\bf y}=0$ has positive dimension, thus we can apply Lemma \ref{Siegel} and \eqref{condition2} to obtain a non-zero solution ${\bf c}=(c_{I})_{I\in \mathcal{R}}\in \mathcal{O}_{K}^{R}$ such that
\begin{equation}
H(1:{\bf c})\lesssim_{K} \left(R(cN_{1})^{D}\right)^{\frac{4d^{2}|\tilde{\mathcal{C}}|}{R-2d^{2}|\tilde{\mathcal{C}}|}}\lesssim_{K}(R(cN_{1})^{D})^{4d^{2}}.
\label{coef complex}
\end{equation}
The solution ${\bf c}=(c_{I})_{I\in \mathcal{R}}$  gives a non-zero homogeneous polynomial $P(Y)=\sum_{I\in \mathcal{R}}c_{I}Y^{I}$ of degree $D$, that vanishes on $\tilde{\mathcal{C}}$, and such that the coefficients verify the bound \eqref{coef complex}. Using the bounds $D\lesssim_{n}r^{1/n}$ and \eqref{upper bound for R}, we conclude
\begin{equation}
H(1:{\bf c})\lesssim_{K, n,m}r^{4d^{2}}N_{1}^{c(K,n)r^{\frac{1}{n}}}
\end{equation}
for some positive constant $c(K,n)$ dependent on $K$ and $n$. Taking $N$ sufficiently large enough, depending on $\kappa,n,m,K,Z$, from the above inequality we deduce
\begin{equation}
H(1:{\bf c})\lesssim_{K,n,m,\kappa}N_{1}^{c'(K,n)r^{\frac{1}{n}}}
\label{polycoef}
\end{equation}
for some positive constant $c'(K,n)$ dependent on $K$ and $n$. Since $N_{1}=c(Z)N$, we conclude that the polynomial $P$ is non-zero, has coefficients in $\mathcal{O}_{K}$, vanishes at $\tilde{\mathcal{C}}$, and moreover it has complexity $\lesssim_{K,n,m,\kappa}\log(c(Z))r^{\frac{1}{n}}=\eta^{\frac{1}{n}}\log(c(Z))(\log(N_{1}))^{\frac{\kappa}{n-\kappa}}$ in $[N]_{\mathcal{O}_{K}}^{n+1}$, where the last inequality is by Definition \ref{r}. Recalling that $Y_{i}=F_{i}(T_{0},\ldots, T_{m})$ is a linear polynomial with \mbox{coefficients} in $\mathcal{O}_{K}$, we conclude that the polynomial 
$$Q=P(F_{0}(T_{0},\ldots ,T_{m}),\ldots ,F_{n}(T_{0},\ldots ,T_{m}))$$
is a non-zero polynomial with coefficients in $\mathcal{O}_{K}$, it vanishes at $\mathcal{C}$, it is non-identically zero at $Z$, and has complexity $\lesssim_{K,n,m,\kappa}\eta^{\frac{1}{n}}\log(c(Z))(\log(N))^{\frac{\kappa}{n-\kappa}}$ in $Z(K,N)$.

Now, we want to prove that $Q$ vanishes in the larger set $X'$ of Proposition \ref{characteristicproposition2}. This will be implied by the vanishing of $Q$ at $\mathcal{C}$ upon choosing adequate constants $\eta,\tau$. For this to be done, we will need to have a control of the size of the image of the polynomial $Q$ in $X'$. Note that since $H(Q({\bf x}))$ depends on the representation ${\bf x}=(x_{0}:\ldots :x_{n})$, we need to choose adequate coordinates. This is done in the following way. Given ${\bf x}\in [N]_{\mathbb{P}^{m}(K)}$, choose coordinates $(x_{0}:\ldots :x_{m})$ as in Lemma \ref{lemita}, and denote ${\bf x}'=(x_{0},\ldots ,x_{m})$ the corresponding affine point. Now, define ${\bf y}'$ to be the affine point ${\bf F}({\bf x}')=(F_{0}({\bf x}'),\ldots ,F_{n}({\bf x}'))$. The choice of coordinates of ${\bf x}$ and \eqref{polynomialheight2} give 
\begin{equation}
H(1:{\bf y}')\leq c(Z) H(1:{\bf x})\leq c(Z)c N=cN_{1}.
\label{poly}
\end{equation}
Applying \eqref{polynomialheight2} to our polynomial $Q$ at the point ${\bf x}'$, and using \eqref{polycoef} and \eqref{poly} we obtain
\begin{equation}
H(Q({\bf x}'))=H(P({\bf y}'))\leq R H(1:{\bf c})H(1:{\bf y}')^{D}\lesssim_{K,n,m,\kappa}N_{1}^{c''(K,n)r^{\frac{1}{n}}}
\label{polycoef2}
\end{equation}
for some positive constant $c''(K,n)$ dependent on $K$ and $n$. In particular, because of \eqref{coefbound}, we have for any ${\bf x}\in [N]_{\mathbb{P}^{m}(K)}$,
\begin{center}
$\log(\mathcal{N}_{K}(Q({\bf x}')))\lesssim_{K,n,m,\kappa}r^{\frac{1}{n}}\log(N_{1})\lesssim_{K,n,m,\kappa}\eta^{\frac{1}{n}}\log(c(Z))\left(\log(N)\right)^{\frac{n}{n-\kappa}}$.
\end{center}
If ${\bf x}\in X' $ and $Q({\bf x})\neq 0$, then $Q({\bf x}')\neq 0$ and we have
\begin{align}
\displaystyle \sum_{\mathfrak{p}\in \mathcal{P}_{I,K}}1_{\mathfrak{p}|Q({\bf x}')}\log(\mathcal{N}_{K}(\mathfrak{p})) & \leq \log\left(\displaystyle\prod_{\mathfrak{p}|Q({\bf x}')}\mathcal{N}_{K}(\mathfrak{p})\right)\leq \log(\mathcal{N}_{K}(Q({\bf x}')))\label{intermediate}\\ & \lesssim_{K,n,m,\kappa}\eta^{\frac{1}{n}}\log(c(Z))(\log(N))^{\frac{n}{n-\kappa}}\nonumber.
\end{align}
Upon choosing adequately $\eta$ and $\tau$, we will see that \eqref{intermediate} does not happen.

Let again ${\bf x}\in X'$. Let $\mathfrak{p}$ be a prime ideal that contributes to the left hand side in the sum of Proposition \ref{characteristicproposition2}. Then there exists ${\bf z}\in \mathcal{C}$ such that ${\bf x}\equiv {\bf z}\mod (\mathfrak{p})$. Since $Q$ vanishes at $\mathcal{C}$, we have $Q({\bf x})\equiv Q({\bf z})=0\mod (\mathfrak{p})$, so we must have $\mathfrak{p}|Q({\bf x})$. We conclude that every prime ideal $\mathfrak{p}$ that contributes to the left hand side in the sum of Proposition \ref{characteristicproposition2} also contributes to the left hand side of \eqref{intermediate}. Then from Proposition \ref{characteristicproposition2} we have that the left side of \eqref{intermediate} is at least $\gtrsim_{K}|I|\gtrsim_{K}\tau(\log(N))^{\frac{n}{n-\kappa}}$. Choose $\eta$ and $\tau$ to satisfy
\begin{center}
$\tau\geq C_{2}\log(c(Z))\eta^{\frac{1}{n}}$
\end{center}
for a constant $C_{2}$ large enough, dependent of $K,n,m$, and $\kappa$. Since by Proposition \ref{characteristicproposition2} we also required $\eta\geq C_{1}\alpha\tau^{\kappa}$, it is enough to take
\begin{center}
$\eta\geq \left(C_{1}\alpha C_{2}^{\kappa}(\log(c(Z)))^{\kappa}\right)^{\frac{n}{n-\kappa}}\gtrsim_{K,n,\kappa,m} \alpha^{\frac{n}{n-\kappa}}(\log(c(Z)))^{\frac{\kappa n}{n-\kappa}}.$
\end{center} 
Then $Q({\bf x}')=Q({\bf x})=0$. Since this holds for all ${\bf x}\in X'$, we conclude the proof. 
\end{proof}

\begin{remark}
The constant $c(Z)$ can be made effective if we use an effective version of Noether's normalization. For instance, using \cite[Theorem 4.2]{P2} we may take $c(Z)\lesssim_{K,m}\deg(Z)^{m-\dim(Z)}$, where $\deg(Z)$ is the degree of $Z$.
\label{remark-1}
\end{remark}

\begin{remark}
In the proof of Theorem \ref{affinegeneralization} we found a polynomial $Q$ of complexity at most $\lesssim_{K,n,m,\kappa,\varepsilon}\eta^{\frac{1}{n}}\log(c(Z))(\log(N))^{\frac{\kappa}{n-\kappa}}$ in $Z(K,N)$. Since we chose $\eta\gtrsim_{K,n,\kappa,m}\alpha^{\frac{n}{n-\kappa}}(\log(c(Z)))^{\frac{\kappa n}{n-\kappa}}$, we may construct the polynomial $Q$ such that its complexity is at most $\lesssim_{K,n,m,\kappa,\varepsilon}\alpha^{\frac{1}{n-\kappa}}\log(c(Z))^{\frac{n}{n-\kappa}}(\log(N))^{\frac{\kappa}{n-\kappa}}$ in $Z(K,N)$.
\label{remark}
\end{remark}

We conclude this section by proving that Theorem \ref{affinegeneralization} implies a stronger theorem than Theorem \ref{generalization}.

\begin{corollary}
For all $n>0$, all real $0\leq \kappa<n$ and all global field $K$, there exists $\tau=\tau(n,\kappa,K)\geq 1$ such that the following holds. Denote $I$ for the interval $[\tau(\log(N))^{\frac{n}{n-\kappa}},2\tau(\log(N))^{\frac{n}{n-\kappa}}]$. Write $\mathcal{P}_{I,K}$ for the set of prime ideals $\mathfrak{p}\subseteq \mathcal{O}_{K}$ defined as
\begin{center}
$\mathcal{P}_{I,K}:=\begin{cases}\left\{\mathfrak{p}\subseteq \mathcal{O}_{K}:\mathcal{N}_{K}(\mathfrak{p})\in I\right\}\;\text{if}\;K\;\text{is a number field},\\ \left\{ \mathfrak{p}\subseteq \mathcal{O}_{K}:\mathcal{N}_{K}(\mathfrak{p})= \tau(\log(N))^{\frac{n}{n-\kappa}}\right\}\;\text{if}\;K\;\text{is a function field}.\end{cases}$
\end{center}
For every $X\subseteq [N]_{K}^{n}$, consider the embedding $X\hookrightarrow \mathbb{P}^{n}(K)$ given by ${\bf x}\mapsto (1:{\bf x})$. Denote $\overline{X}$ the image of $X$ by this embedding. If $|\overline{X}_{\mathfrak{p}}|\lesssim \mathcal{N}_{K}(\mathfrak{p})^{\kappa}$ for every prime $\mathfrak{p}\in \mathcal{P}_{I,K}$, then, for every $\varepsilon>0$, there exists some non-zero $P\in \mathcal{O}_{K}[X_{1},\ldots ,X_{n}]$, of complexity $\lesssim_{\kappa,n,\varepsilon,K}(\log(N))^{\frac{\kappa}{n-\kappa}}$ in $[N]_{\mathcal{O}_{K}}^{n}$ vanishing on at least $(1-\varepsilon)|X|$ points of $X$.  
\label{generalization2}
\end{corollary}
\begin{proof}
Let $X\subseteq [N]_{K}^{n}$ be set as in Theorem \ref{generalization}. Let ${\bf x}\in [N]_{K}^{n}$. Then \eqref{inequality of heights0} implies that $H(1:{\bf x})\leq N^{n}$. Consider the embedding $[N]_{K}^{n}\rightarrow [N^{n}]_{\mathbb{P}^{n}(K)}$ given by ${\bf x}\mapsto (1:{\bf x})$. Then $\overline{X}$ satisfies the hypothesis of Theorem \ref{affinegeneralization}, so we can apply this theorem to $\overline{X}$. We obtain a non-zero homogeneous polynomial $P(X_{0},\ldots ,X_{n})$ of the desired complexity, that vanishes on at least $(1-\varepsilon)|\overline{X}|$ points of $\overline{X}$. Now $P(1,X_{1},\ldots ,X_{n})$ is a polynomial that satisfies the conclusion of Corollary \ref{generalization2}. 
\end{proof}

\section{Counting points in transcendental surfaces}
\label{4}

Having proved Theorem \ref{affinegeneralization}, we are going to prove a more general version of Theorem \ref{app}. For the \mbox{corresponding} definitions on o-minimality, see \cite{Dries}. If $X\subseteq \mathbb{R}^{n}$ and $K$ is a number field, let $X(K)$ denote the subset of points with coordinates in the field $K$. For $N\geq 1$, let
\begin{center}
$X(K,N):=\{{\bf x}=(x_{1},\ldots ,x_{n})\in X(K):H(x_{i})\leq N\;\forall i\}$.
\end{center}
Let $X\subseteq \mathbb{R}^{n}$ be a set definable in an o-minimal structure, and let $X^{\text{alg}}$ be the set of points ${\bf x}\in X$ such that there exists a connected, semialgebraic set $Y\subseteq X$ of positive dimension with ${\bf x}\in Y$. We define $X^{\text{trans}}:=X\backslash X^{\text{alg}}$. 

As an attempt to prove Conjecture \ref{wilkie} and its generalization \eqref{wilkiepf}, we pose the following conjectures. In what follows, $\tilde{\mathbb{R}}=(\mathbb{R},<,+,\cdot ,f_{1},\ldots ,f_{r})$ is a model complete expansion of the real field by a Pfaffian chain $f_{1},\ldots ,f_{r}$. Also, if $X\subseteq \mathbb{R}^{n}$, we will denote $\overline{X}$ for the image of $X$ in $\mathbb{P}^{n}(\mathbb{R})$ by the embedding ${\bf x}\mapsto (1:{\bf x})$.

\begin{conjecture}[Ill-distribution Conjecture A] 
Suppose that $X\subseteq \mathbb{R}^{n}$ is a set definable in the o-minimal structure $\tilde{\mathbb{R}}$, of dimension at most $n-1$. Then there are positive constants $\alpha=\alpha(X,K)$, $\tau=\tau(X,K)$, $\kappa=\kappa(X)$, with $0\leq \kappa<n$ such that
\begin{equation}
|\left(\overline{X(K,N)}\right)_{\mathfrak{p}}|\leq \alpha\mathcal{N}_{K}(\mathfrak{p})^{\kappa}
\end{equation}
for all $N>e$ and all prime ideals $\mathfrak{p}$ with $\mathcal{N}_{K}(\mathfrak{p})\geq \tau(\log(N))^{\frac{n}{n-\kappa}}$.
\label{wilkieill}
\end{conjecture}
  
\begin{conjecture}[Ill-distribution Conjecture B]
Suppose that $X\subseteq \mathbb{R}^{n}$ is a set definable in the o-minimal structure $\tilde{\mathbb{R}}$. Then there are positive constants $\alpha=\alpha_{1}(X,K)$, $\tau=\tau(X,K)$, $\kappa=\kappa(X)$, with $0\leq \kappa<n$ such that
\begin{equation}
|\left(\overline{X^{{\rm trans}}(K,N)}\right)_{\mathfrak{p}}|\leq \alpha\mathcal{N}_{K}(\mathfrak{p})^{\kappa}
\end{equation}
for all $N>e$ and all prime ideals $\mathfrak{p}$ with $\mathcal{N}_{K}(\mathfrak{p})\geq \tau(\log(N))^{\frac{n}{n-\kappa}}$.
\label{wilkieill2}
\end{conjecture}  

We note that the condition in the dimension in Conjecture \ref{wilkieill} is because a definable subset $X\subseteq \mathbb{R}^{n}$ of dimension $n$ contains an open disk, thus in this case we have $|(\overline{X(K,N)})_{\mathfrak{p}}|\gtrsim \mathcal{N}_{K}(\mathfrak{p})^{n}$ for all prime ideals $\mathfrak{p}$ of absolute norm large enough.  

It is clear that Conjecture \ref{wilkieill} implies Conjecture \ref{wilkieill2}. Also, note that in general, we can not expect Conjecture \ref{wilkieill} and Conjecture \ref{wilkieill2} to hold for all primes $\mathfrak{p}$. Indeed, take $f(x)=2^{x}$ and consider $X$ the graph of $f$. Then if $K=\mathbb{Q}$ and $p\lesssim \log\log(N)$ is a prime, $|(X(\mathbb{Z},N))_{p}|\sim pu_{p}$, where $u_{p}$ is the order of $2$ in $(\mathbb{Z}/p\mathbb{Z})^{\times}$. Since it is expected that $u_{p}=p-1$ for many primes, we expect $|(X^{\text{trans}}(\mathbb{Z},N))_{p}|=|(X(\mathbb{Z},N))_{p}|\sim p^{2}$ for many primes $p\lesssim \log\log(N)$. 

Now we are going to prove that Conjecture \ref{wilkieill2} is equivalent to Conjecture \ref{wilkiepf} for sets of dimension at most $2$.

\begin{theorem}
Let $X\subseteq \mathbb{R}^{n}$ be a set definable in $\tilde{\mathbb{R}}$ of dimension at most $2$. Then $X$ verifies Conjecture \ref{wilkieill2} if and only if $X$ verifies Conjecture \ref{wilkiepf}.
\label{evidence}
\end{theorem}
\begin{proof}
Note that if $n=1$, both conjectures are trivial. So suppose that $n>1$. We begin by proving that Conjecture \ref{wilkiepf} implies Conjecture \ref{wilkieill2} without restriction in the dimension of $X$.

 Suppose that Conjecture \ref{wilkiepf} holds for some set $X\subseteq \mathbb{R}^{n}$ with $n>1$ definable in $\tilde{\mathbb{R}}$. We may suppose that $c_{2}\geq \frac{n}{n-1}$. Take $\kappa$ so that $\frac{n}{n-\kappa}=c_{2}$ and $\kappa\geq 1$. Then if $\mathfrak{p}$ is a prime in $\mathcal{O}_{K}$ such that $\mathcal{N}_{K}(\mathfrak{p})\geq c_{1}\log(N)^{\frac{n}{n-\kappa}}=c_{1}\log(N)^{c_{2}}$, we have that 
\begin{center}
$|\left(\overline{X^{\text{trans}}(K,N)}\right)_{\mathfrak{p}}|\leq |X^{\text{trans}}(K,N)|\leq c_{1}(\log(N))^{c_{2}}\leq \mathcal{N}_{K}(\mathfrak{p})\leq \mathcal{N}_{K}(\mathfrak{p})^{\kappa}$. 
\end{center}
Taking $\alpha=1$, $\kappa=\frac{c_{2}-1}{c_{2}}n$ and $\tau\geq c_{1}$, we conclude that $X$ satisfies Conjecture \ref{wilkieill2}.

For the other implication, let us suppose that $X\subseteq \mathbb{R}^{i}$ is a set of dimension $i-1$, where $i=2,3$, satisfying Conjecture \ref{wilkieill2}. It is showed in the proof of \cite[Theorem 5.4]{Jones} that it is enough to suppose that $X$ is the graph of an implicitly definable function defined on an open cell in $\mathbb{R}^{i-1}$ (for the definition of an implicitly definable function, see \cite[Page 645]{Jones}). For such $X$,  we can apply Corollary \ref{generalization2} with $\varepsilon=\frac{1}{2}$ to $X^{\text{trans}}(K,N)$, to find a non-zero polynomial $Q\in \mathcal{O}_{K}[T_{j}]_{1\leq j\leq i}$ of degree at most $\lesssim_{\kappa,K,n}(\log(N))^{\frac{\kappa}{n-\kappa}}$ vanishing on at least half the points of $X^{\text{trans}}(K,N)$. This means that
\begin{equation}
\left |X^{\text{trans}}(K,N)\cap V(Q)\right|\geq \dfrac{1}{2}\left| X^{\text{trans}}(K,N) \right|.
\label{bound1}
\end{equation}
Now, notice that
\begin{equation}
X^{\text{trans}}(K,N)\cap V(Q)\subseteq (X\cap V(Q))^{\text{trans}}(K,N).
\label{bound2}
\end{equation}
By \cite[Lemma 3.3]{Jones} if $i=2$, and \cite[Proposition 5.3]{Jones} if $i=3$, we have for any non-zero polynomial $P\in \mathbb{R}[X_{j}]_{1\leq j\leq i}$ of degree $d$ the bound
\begin{equation}
\left|(X\cap V(P))^{\text{trans}}(K,N)\right|\leq c_{1}(X,K)d^{c_{2}(X)}(\log(N))^{c_{3}(X)},
\label{bound3}
\end{equation}
where $c_{1}(X,K),c_{2}(X),c_{3}(X)$ are positive effective constants. Applying \eqref{bound3} for the polynomial $Q$ we constructed, and using \eqref{bound1} and \eqref{bound2} we conclude that X satisfies Conjecture \ref{wilkiepf}. The general case follows by an argument with projections, as the one explained in \cite[Page 644]{Butler}.
\end{proof}

\begin{remark}
We already noted that Conjecture \ref{wilkieill} implies Conjecture \ref{wilkieill2}. Hence, the proof of Theorem \ref{evidence} also shows that Conjecture \ref{wilkieill} for subsets $X\subseteq \mathbb{R}^{n}$ definable in $\tilde{\mathbb{R}}$ of dimension at most $2$ implies Conjecture \ref{wilkiepf} for such subsets. 
\end{remark}

We do not know if Conjecture \ref{wilkieill} implies Conjecture \ref{wilkiepf}. Similarly, in general it is still unknown if the existence of mild parametrizations for sets definable in $\tilde{\mathbb{R}}$ implies Conjecture \ref{wilkiepf}. However, in \cite[Conjecture 3.5]{Pila4} Pila proves that, given $X\subseteq \mathbb{R}^{n}$ a subset definable in $\tilde{\mathbb{R}}$,  the existence of ``uniform'' mild parametrizations for the family of sets $\{X\cap V\}_{V\in \mathcal{F}}$, with $\mathcal{F}$ an algebraic family of algebraic varieties, implies Conjecture \ref{wilkiepf}. Motivated by this conjecture of Pila, we pose the following conjecture:

\begin{conjecture}
Suppose that $X\subseteq \mathbb{R}^{n}$ is a subset definable in the o-minimal structure $\tilde{\mathbb{R}}$. Let $d>0$ and $r>0$ be positive integers. There are positive constants $\alpha=\alpha(X,K)$, $c_{1}=c_{1}(X,K)$, $c_{2}=c_{2}(X,K)$, $\tau=\tau(X,K)$, $\kappa=\kappa(X,r)$ with $0\leq \kappa<r$, such that for any algebraic variety $V$  defined over $\mathbb{R}$ of dimension $r$ and degree bounded by $d$, we have
\begin{equation}
\left\vert\left( \overline{(X\cap V)^{\text{trans}}(K,N)}\right)_{\mathfrak{p}}\right\vert\leq \alpha d^{c_{1}}\mathcal{N}_{K}(\mathfrak{p})^{\kappa}
\end{equation}
for all $N>e$ and all prime ideals $\mathfrak{p}$ with $\mathcal{N}_{K}(\mathfrak{p})\geq \tau d^{c_{2}}(\log(N))^{\frac{n}{n-\kappa}}$.
\label{wilkieill3}
\end{conjecture}

We believe that Conjecture \ref{wilkieill3} implies Conjecture \ref{wilkiepf}, but we do not have a proof. It is most likely that, given the inductive nature of Conjecture \ref{wilkiepf}, in order to obtain Conjecture \ref{wilkiepf} from Conjecture \ref{wilkieill3} one requires a stronger version for Theorem \ref{affinegeneralization}, where the ambient variety $Z$ is not required to contain the set $X$ and the bound in the polynomial that we find gets better as the degree of the ambient variety $Z$ gets larger (this is known to be possible in other versions of the polynomial method; for instance, see \cite{Walsh4}).

\subsection*{Acknowledgements}  The author is very grateful to his advisor Rom\'an Sasyk for a careful reading of the manuscript and several helpful discussions. The author would also like to thank Juan Menconi for his useful comments. The author is very grateful to the referee for his/her comments and for carefully reading the article. This work was partially supported by a CONICET doctoral fellowship.

\end{document}